\documentclass[10pt]{amsart}

\usepackage{amsopn}
\usepackage{amsmath, amsthm, amssymb, amscd, amsfonts, xcolor}
\usepackage[latin1]{inputenc}
\usepackage{xcolor}
\usepackage[pdftex]{hyperref}  

\theoremstyle{plain}
\newtheorem{mainteo}{Theorem}

\newtheorem{teo}{Theorem}[section]
\newtheorem{cor}[teo]{Corollary}

\newtheorem{prop}[teo]{Proposition}
\newtheorem{lema}[teo]{Lemma}
\newtheorem*{assumption}{Assumption}
\newtheorem*{sublemma}{Sublemma}

\theoremstyle{definition}
\newtheorem{defi}[teo]{Definition}
\newtheorem{ejemplo}[teo]{Example}

\newtheorem*{notation}{Notation}

\theoremstyle{remark}
\newtheorem{nota}[teo]{Remark}

\numberwithin{equation}{section}

\newcommand\Ad{\operatorname{Ad}}

\newcommand\Aff{\operatorname{Aff}}

\newcommand\diag{\operatorname{diag}}

\newcommand\Exp{\operatorname{Exp}}

\newcommand\Lie{\operatorname{Lie}}

\newcommand\SO{\operatorname{SO}}

\newcommand\Spin{\operatorname{Spin}}
\newcommand\SU{\operatorname{SU}}

\newcommand\Tr{\operatorname{Tr}}

\newcommand\so{\mathfrak{so}}
\newcommand\tr{\mathfrak{tr}}
\renewcommand\gg{\mathfrak{g}}

\newcommand\gh{\mathfrak{h}}

\newcommand\gm{\mathfrak{m}}

\makeatletter\renewcommand\theenumi{\@roman\c@enumi}\makeatother

\title{The index of symmetry of compact naturally reductive spaces}

\author{Carlos Olmos}

\author{Silvio Reggiani}

\address{Facultad de Matem\'atica, Astronom\'ia y F\'isica,
  Universidad Nacional de C\'ordoba, Ciudad Universitaria, 5000
  C\'ordoba, Argentina}
\email{\href{mailto:olmos@famaf.unc.edu.ar}{olmos@famaf.unc.edu.ar},
  \href{mailto:reggiani@famaf.unc.edu.ar}{reggiani@famaf.unc.edu.ar}}

\author{Hiroshi Tamaru}

\address{Department of Mathematics, Hiroshima University, 1-3-1
  Kagamiyama, Higashi-Hi\-ro\-shi\-ma, 739-8526, Japan} 
\email{\href{mailto:tamaru@math.sci.hiroshima-u.ac.jp}{tamaru@math.sci.hiroshima-u.ac.jp}}



\thanks{The work of C.~Olmos and S.~Reggiani was supported by
  Universidad Nacional de C\'ordoba and CONICET, and partially
  supported by ANCyT, Secyt-UNC and CIEM. H.~Tamaru was supported in
  part by KAKENHI (24654012)} 

\begin{document}

\begin {abstract}
We introduce a geometric invariant that we call the index of symmetry,
which measures how far is a Riemannian manifold from being a symmetric
space. We compute, in a geometric way, the index of symmetry of
compact naturally reductive spaces. In this case, the so-called leaf
of symmetry turns out to be of the group type. We also study several
examples where the leaf of symmetry is not of the group
type. Interesting examples arise from the unit tangent bundle of the
sphere of curvature $2$, and two metrics in an Aloff-Wallach
$7$-manifold and the Wallach $24$-manifold. 
\end {abstract}

\maketitle

\section{Introduction}

The study of Riemannian homogeneous spaces is a very important
research area inside Riemannian geometry.  The symmetric spaces,
defined and classified by \'E.~Cartan \cite{cartan-1926-1927}, are a
distinguished family among all homogeneous spaces.  The symmetric
spaces can be defined in several ways.  For example, these spaces 
are locally characterized by the property of having parallel curvature
tensor, or globally, by the fact that the geodesic symmetry at any
point extends to a global isometry.  Another way of defining a
symmetric space $M$ is the following.  Given $p \in M$ and $v \in
T_pM$ there exists a Killing field $X$ on $M$ such that $X(p) = v$ and
$(\nabla X)_p = 0$.  That is, there is a Killing field, in any
direction, which is parallel at $p$. 

The symmetric spaces generalize, in a natural way, to larger families
of homogeneous spaces.  This is the case of naturally reductive
spaces.  Recall that in a naturally reductive space $M = G/H$ there
always exists a canonical connection $\nabla^c$, which is
$G$-invariant and has the same geodesics as the Levi-Civita
connection (and hence, $\nabla^c$ has totally skew-symmetric torsion).
In particular, the Riemannian curvature tensor is parallel with
respect to the canonical connection, $\nabla^c R = 0$.  For a
symmetric space, the Levi-Civita connection is a canonical
connection.  

In recent years, some relevant results on naturally reductive spaces
have been obtained in the framework of connections with skew-symmetric
torsion.  In particular, it is proved that the canonical connection of
a naturally reductive space is essentially unique \cite{OR, OR2}
(except for certain well-studied cases, which are all symmetric). 

Notice that the family of naturally reductive spaces contains the
compact isotropy irreducible spaces and, more generally, the normal
homogeneous spaces. 

In this article we deal with a geometric invariant $0 \le i_{\mathfrak
  s}(M) \le \dim M$, that we call the \emph{index of symmetry} of the
Riemannian manifold $M$.  Roughly speaking, the index of symmetry of
$M$ measures how far is $M$ from being a symmetric space, in the sense
that $M$ is symmetric if and only if $i_{\mathfrak s}(M) = \dim M$.
The index of symmetry of $M$ can be defined as follows: it is the
greatest non-negative integer $k$ such that at every $p \in M$ there exist
at least $k$ linear independent vectors $v_1, \ldots, v_k \in T_pM$
and $k$ Killing fields $X_1, \ldots, X_k$ on $M$ such that $X_i(p) =
v_i$ and $(\nabla X_i)_p = 0$. 

If $i_{\mathfrak s}(M) = k$, given $p \in M$, one can consider the
subspace 
$$\mathfrak s_p = \{X(p): \text{$X$ is a Killing field on $M$ with
  $(\nabla X)_p = 0$}\} \subset T_pM,$$ 
which is a subspace of dimension at least $k$.  So, $p \mapsto
\mathfrak s_p$ defines an, a priori non-smooth nor even of constant
dimension, distribution $\mathfrak s$ on $M$.  The distribution
$\mathfrak s$ is called the \emph{distribution of symmetry} of $M$. 

If $M$ is a Riemannian homogeneous space, we have that the
distribution of symmetry has constant dimension and it is smooth
(since it is invariant under the full isometry group of $M$).  Hence
$i_{\mathfrak s}(M) = \dim \mathfrak s_p$ for any $p \in M$.  We prove
that $\mathfrak s$ is an integrable distribution with totally geodesic
leaves (or equivalently, $\mathfrak s$ is autoparallel).  Moreover,
the leaves of $\mathfrak s$ turn out to be intrinsically globally
symmetric submanifolds of $M$. 

The main goal of this article is to compute the index of symmetry of
compact normal homogeneous spaces.  More generally, we explicitly
determine the distribution of symmetry for these spaces. Namely, 

\begin{mainteo}\label{main1} 
Let $M = G/H$ be a simply connected compact  normal homogeneous space,
where $G$ is connected.  Let us assume that $M$ is an irreducible
Riemannian manifold which is not a symmetric space.  Then the
distribution of symmetry of $M$ coincides with the $G$-invariant
distribution defined by the fixed vectors of $H$.  
\end{mainteo}

The above theorem is still true for compact naturally reductive spaces
under the hypothesis that the transitive presentation group $G$ is the
transvection group of the canonical connection. Namely, 

\begin{mainteo}\label{main2} 
Let $M = G/H$ be a simply connected compact naturally reductive space,
where $G$ is the group of transvections of the canonical connection.
Let us assume that $M$ is an  irreducible Riemannian manifold which is
not a symmetric space.  Then the distribution of symmetry of $M$
coincides with the $G$-invariant distribution defined by the fixed
vectors of $H$. 
\end {mainteo}

In order to prove Theorem \ref{main1}, we deal with the de Rham
decomposition of the leaf of symmetry $L(p)$ of $\mathfrak{s}$ at an
arbitrary point $p$.  We first prove that the (connected component by
$p$) of the fixed points of $H$ in $M$, $\Sigma $, must be contained
in $L(p)$.  Moreover, we show that the flat factor of $L(p)$ by $p$
must be inside $\Sigma $.  Then, we finally prove that the semisimple
factors of $L(p)$ must be tangent to $\Sigma$.  The proof depends on
general arguments and uses strongly Proposition \ref{bracket formula}
and the precise knowledge of the full isometry group (see \cite{R, OR,
  OR2}). 

For proving Theorem \ref{main2} we make use of the so-called Kostant
bilinear form, which allows us to think of a naturally reductive space
as a normal homogeneous space with respect to a bi-invariant
pseudo-Riemannian metric defined on the group of presentation.  In
this case, the presentation group must be the group of transvection of
the canonical connection.  Similar arguments, as for the normal
homogeneous case, work.  (Though it is not trivial to adapt some of
the arguments.)

Recall, as it follows from Theorems \ref{main1} and \ref{main2}, that
for naturally reductive spaces the leaf of symmetry is of the group
type.  We finish the article by giving several examples of compact
homogeneous spaces with non-trivial index of symmetry and such that the
leaf of symmetry is not of the group type.  To do this, we work with
triples $G \supset G' \supset K'$, where $G/G'$ and
$G'/K'$ are symmetric spaces.  The examples arise by perturbing the
normal homogeneous metric on $G/K'$.  Interesting examples, 
obtained in this way, are the unit tangent bundle of the sphere of
curvature $2$, and two metrics that occur in the Aloff-Wallach
manifold $W^7_{1, -1} = \SU(3)/\SO(2)$ and the Wallach manifold
$W^{24} = F_4/\Spin(8)$. 
 
This article can be regarded as an effort to understand, in a
geometric way, naturally reductive spaces.  This is in the same spirit
of the articles \cite {OR, R, OR2, R2}.

\section{Preliminaries and basic facts}
\label{preliminaries}

\subsection{Infinitesimal transvections}

Let $M$ be a Riemannian manifold.  Throughout this article we will
call an \emph{infinitesimal transvection}, or just a \emph{(geometric)
  transvection}, at $p \in M$, to a Killing vector field $X$ such that
$(\nabla X)_q  =0$, where $\nabla$ is the Levi-Civita connection.
Recall that $M^n$ a symmetric space if and only if for all $p \in M$
there exist transvections $X_1, \ldots, X_n$, at $p$, such that
$X_1(p), \ldots, X_n(p)$ is a basis of $T_pM$.

Do not confuse the infinitesimal transvection with the transvection of
the canonical connection of a naturally reductive space.  In fact, let
$M = G/H$ be a naturally reductive space with associated reductive
decomposition $\gg = \gh \oplus \gm$.  That is, $M$ carries a
$G$-invariant metric, $\gg$ is the Lie algebra of $G$, $\gh$ is the
Lie algebra of $H$ and $\gm$ is an $\Ad(H)$-invariant subspace such
that the geodesics through $p = eH$ are given by $\Exp(tX) \cdot p$,
$X \in \gm$. Let $\nabla^c$ be the canonical connection of $M$ (i.e.,
the $G$-invariant connection associated with the decomposition $\gg =
\gh \oplus \gm$).  The Lie algebra of transvections of $\nabla^c$ is
$\tr(M, \nabla^c) = \gm + [\gm, \gm]$, and the connected associated Lie
subgroup is $\Tr(M, \nabla^c)$. Recall that $\Tr(M, \nabla^c)$ is a
transitive normal subgroup of $G$.  When the metric on $M$ is also
normal homogeneous, then $\Tr(M, \nabla^c)$ coincides with connected
component of $G$ (see \cite{R}).

Of course, a transvection of the canonical connections needs not to be
a geometric transvection.

\subsection{Ad-invariant bilinear forms}

In this subsection we want to point out the following elementary
remark, which will be very useful in the sequel.

\begin{nota}\label{bilinear}
Let $\mathfrak g$  be a Lie algebra and  $Q$ an $\Ad$-invariant
symmetric bilinear form on $\mathfrak g$.  Assume that $\mathfrak g$
is the sum of the ideals $\mathfrak g = \mathfrak g_1 \oplus \mathfrak
g_2$ where $\mathfrak g_1$ is semisimple.  Then, such a decomposition
must be orthogonal with respect to $Q$, that is $Q(\mathfrak g_1,
\mathfrak g_2) = 0$.  Moreover, if $\mathfrak g_1$ is simple, then the
restriction $Q|_{\mathfrak g_1}$ of $Q$ to the ideal $\mathfrak g_1$
must be a scalar multiple of the Killing form of $\mathfrak g_1$. 

In fact, let $X', X'' \in \mathfrak g_1$ and $Y \in \mathfrak g_2$.  If
$X = [X', X'']$, an standard calculation gives 
$$Q(X, Y) = Q([X', X''], Y) = -Q(X'', [X', Y]) = -Q(X'', 0) = 0.$$ 
Since $\mathfrak g_1$ is semisimple, it is linearly spanned by
elements of the form $X = [X', X'']$, and therefore $Q(\mathfrak g_1,
\mathfrak g_2) = 0$. 

If we also assume that $\mathfrak g_1$ is simple, it is well known
that Schur's Lemma implies that $Q|_{\mathfrak g_1}$ is a multiple of
the Killing form of $\mathfrak g_1$. 
\end{nota}

\subsection{The isometry group of naturally reductive spaces}

Let $M = G/H$ be a compact and locally irreducible naturally reductive
space.  Let $\nabla^c$ be the canonical connection on $M$.  In
\cite{OR} it is proved that the connected component of the group
$\Aff_0(M, \nabla^c)$ of $\nabla^c$-affine transformations coincides
with the connected component of the full isometry group of $M$, except
for spheres or real projective spaces.  On the other hand, in \cite{R}
it is studied the structure of $\Aff_0(M, \nabla^c)$ when $M$ is also
normal homogeneous.  This gives the following characterization of the
(connected component of the) isometry group of a normal homogeneous
space. 

\begin{teo}[see \cite{R}]\label{isometry group}
Let $M = G/H$ be a compact normal homogeneous space.  Assume that $M$
is locally irreducible and that it is neither (globally) isometric to a
sphere, nor to a real projective space.  Write $G = G_\mathrm{ab}
\times G_\mathrm{ss}$ as an almost direct product, where
$G_\mathrm{ab}$ is abelian and $G_\mathrm{ss}$ is a semisimple Lie
group of the compact type.  Then
$$I_0(M) = G_\mathrm{ss} \times F \qquad \text{(almost direct
  product)},$$
where $F$ is the connected component by $p = eH$ of the fixed points
of $H$ (regarded as a Lie group).
\end{teo} 

Recall that the Lie algebra of $F$ may be identified with the Lie
algebra of $G$-invariant fields on $M$. 

Actually, Theorem \ref{isometry group} remains true if we assume that
$M$ is naturally reductive.  In fact, the key factor to prove Theorem
\ref{isometry group} in \cite{R} is that $G$-invariant fields are
Killing fields.  This is also true if $M$ is naturally reductive.  In
fact, the difference tensor $D = \nabla - \nabla^c$, where $\nabla$ is
the Levi-Civita connection of $M$, is totally skew-symmetric.  If $X$
is a $G$-invariant field, then $\nabla^c X = 0$ (since the canonical
connection is $G$-invariant).  Then
$$\nabla X = \nabla X - \nabla^cX = DX$$
is skew-symmetric, and hence $X$ is a Killing field.  Therefore, the
proof of Theorem~\ref{isometry group} given in \cite{R} also works in
the naturally reductive case (see \cite{OR2}).

\section{The index of symmetry}  

Let $M^n$ be a Riemannian manifold and denote by $\mathfrak K(M)$ the
Lie algebra of global Killing fields on $M$.  For $q \in M$, let us
define the Cartan subspace $\mathfrak p^q$ at $q$ by 
$$\mathfrak p^q := \{ X \in \mathfrak K(M): (\nabla X)_q = 0\}.$$ 
The symmetric isotropy algebra at $q$ is defined by 
$$\mathfrak k^q := \{[X , Y]: X, Y \in \mathfrak p^q\}.$$ 
Observe that $\mathfrak k ^q$ is contained in the (full) isotropy
subalgebra $\mathfrak K_q(M)$.  In fact, if $X, Y \in \mathfrak p^q$,
then $[X, Y]_q = (\nabla_XY)_q - (\nabla_YX)_q = 0$.  Moreover, since
$\mathfrak p^q $ is left invariant by the isotropy at $q$, 
$$\mathfrak g^q := \mathfrak k^q \oplus \mathfrak p^q$$ 
is an involutive Lie algebra. 

\begin{nota}\label{transporte} 
If $X \in \mathfrak p^q$, then $ \gamma(t) = \Exp(tX) \cdot q$ is a
geodesic.  Moreover, the parallel transport along $\gamma(t)$ is given
by $dm_{\Exp(tX)}|_q$, where $m_g(x) = g \cdot x$.  In fact, for any
Killing field $X$,  
$$\tau_{-t} \circ dm_{\Exp(tX)}|_q  = e^{t(\nabla X)_q} \in
\so(T_qM),$$
where $\tau_t$ is the parallel transport along the curve $\Exp(tX)
\cdot q$ (see \cite[p.~163]{BCO}).
\end{nota}

Let $G^q$ be the Lie subgroup, of the full isometry group $I(M)$,
associated to the Lie algebra $\mathfrak g^q$.  By the previous
remark, the orbit $G^q \cdot q$ is a totally geodesic submanifold of $M$. 

Observe that $G^{g \cdot q} = g \cdot G^q \cdot g^{-1}$ and that $G^x = G^q$, for
all $x\in G^q \cdot q$.

The symmetric subspace at $q$, $\mathfrak s_q \subset T_qM$, is
defined by 
$$\mathfrak s_q := \{X \cdot q : X \in \mathfrak p^q\}  = \mathfrak
p^q \cdot q.$$

\begin{defi} The \textit{index of symmetry} $i_{\mathfrak s}(M)$ of a
  Riemannian manifold $M$ is the infimum, over $q \in M$, of the
  dimensions of $\mathfrak s_q$. 
\end{defi}

Observe that
$$L(q) := G^q \cdot q = \exp(\mathfrak s_q).$$
Moreover, for any $x \in G^q \cdot q$, $T_x(G^q \cdot q) = \mathfrak s_x$.
So, the totally geodesic submanifold $L(q)$ is a leaf of the a priori
non necessarily smooth (and eventually singular) distribution $p
\mapsto \mathfrak s_p$, $p\in M$.  Note that $L(q)$ is a locally
symmetric space, whose group of transvections is $G^q$, quotioned by
the ideal of elements that acts trivially on $L(q)$ (we will see later
that  $G^q$ acts almost effectively on $L(q)$ if $M$ is compact).
Moreover, $L(q)$ is a globally symmetric space, as follows from
Theorem A of \cite{EO} (see also Lemma 5 of this reference). 

\begin{lema}\label{effective} 
If $M$ is compact then $G^q$ acts almost effectively on $L(q) = G^q
\cdot q$, for all $q \in M$. 
\end{lema} 

\begin{proof}
Since $M$ is compact, the isometry group $I(M)$ is compact.  Let
$(\cdot, \cdot)$ be an $\Ad(I(M))$-invariant inner product on the Lie
algebra of $I(M)$. 

Let $\mathfrak h \subset \mathfrak k^q$ the ideal which corresponds to
the normal subgroup of $H$ of $G^q$ that acts trivially on $L(q)$.
Let $Z \in \mathfrak h$.  Then $[Z, \mathfrak p^q] \subset \mathfrak
p^q$, since the isotropy at $q$ leaves invariant $\mathfrak p^q$.  On
the other hand, since $H$ acts trivially on $T_q(L(q)) = \mathfrak
s_q$, one has that $[Z,X]_q =0$, for all $X \in \mathfrak p^q$.  Then
$[Z, \mathfrak p^q] = \{0\}$.  Therefore, if $X, Y \in \mathfrak p^q$,
$(Z, [X, Y]) = -([Z, X], Y) = 0$.  Thus, $Z$ is perpendicular to
$[\mathfrak p^q, \mathfrak p^q] = \mathfrak k^q$.  Hence $\mathfrak h
= \{0\}$. 
\end{proof}

We identify $T_q(G^q \cdot q) = \mathfrak s_q \simeq \mathfrak p^q$ and
decompose $\mathfrak p^q = \mathfrak p_0^q \oplus \mathfrak p_1^q
\oplus \cdots \oplus \mathfrak p_s^q$, where $\mathfrak p_0^q$
corresponds to the Euclidean factor and $\mathfrak p_i^q$, $i \ge 1$,
corresponds to the irreducible $i$-th factor in the de Rham
decomposition of $L(q) = G^q \cdot q$.  Let $\mathfrak k_j^q = [\mathfrak
  p_j^q, \mathfrak p_j^q]$ for $j = 0, \ldots, s$ and let $\mathfrak
g_j^q = \mathfrak k_j^q \oplus \mathfrak p_j^q$. 

\begin{cor}\label{ideals}
If $M$ is compact then $\mathfrak k_0^q = \{0\}$, $[\mathfrak g_i^q,
  \mathfrak g_j^q] = \{0\}$, if $i \neq j$ and so $\mathfrak g^q$ is
the direct sum of the ideals $\mathfrak g_1^q, \ldots, \mathfrak
g_s^q$.  Then 
$$G^q = G^q_0 \times \cdots \times G^q_s \qquad \text {(almost direct
  product)},$$ 
where $\Lie(G^q_i) = \mathfrak g_i^q$.
\end{cor}

In the notation of the above corollary, we denote
$$L_i(q) := G^q_i \cdot q, \qquad i= 1, \ldots, s,$$
and we call $L_i(q)$ the \textit{$i$-th de Rham factor} by $q$ of
$L(q)$. 

More generally, if  $J \subset \{0, \ldots, s\}$ we denote by 
$$G^q_J = \prod_{j \in J} G_j^q.$$ 
The orbit $L_J(q) := G_J^q \cdot q$ is called a \textit{local factor by
  $q$} of $L(q)$.

\medbreak

Let 
$$\tilde G^q = \{g \in I(M): g \cdot L(q) = L(q)\}_0.$$

Let $g \in \tilde G^q$, then $G^q = G^{g \cdot q} = gG^qg^{-1}$. So, $G^q$
is a normal subgroup of ~ $\tilde G^q$. 

Let
$$\bar H^q = \{g \in I(M): g \text { acts trivially on } L(q)\}_0.$$ 
Observe that $\bar H^q$ is also a normal subgroup of $\tilde G^q$.  

\begin{lema}\label{direct} 
$\tilde G^q = G^q \times \bar H ^q$ (almost direct product).
\end{lema} 

\begin{proof}
By Lemma \ref{effective}, $G^q \cap \bar H ^q$ is discrete.  Let $X$
be a Killing field induced by $\tilde G^q$.  Then $X|_{L(q)}$ is an
intrinsic Killing field of $L(q)$ which is bounded.  Then it must lie
in the Lie algebra of the (intrinsic) transvection group of $L(q)$.
Then there exists a Killing field $Y \in \mathfrak g^q$ such that
$Y|_{L(q)} = X|_{L(q)}$.  Then $Z = Y - X$ is null when restricted to
$L(q)$.  Then $Z$ lies in the Lie algebra of $\bar H^q$.  This implies
the desired decomposition. 
\end{proof}

We denote, if $J \subset \{0, \ldots, s\}$, by 
\begin{equation}\label{tilde G_J}
\tilde G_J^q = \{g \in I(M): g \cdot L_J(q) = L_J(q)\}_0.
\end{equation}
In the same way as before, one has that $G_J^q$ is a normal subgroup
of $\tilde G_J^q$.  Let
$$\bar H_J^q = \{g \in I(M): g \text { acts trivially on }
L_J(q)\}_0.$$
One has that $\bar H_J^q$ is also a normal subgroup of $\tilde
G_J^q$.  Observe that $G_i^q$ acts trivially on $L_J(q)$, if  $i
\notin J$.  Then, by Lemma \ref{direct}, 
$$\bar H_J^q = \bar H^q \times \hat G_J ^q$$
where 
$$\hat G_J^q = \prod_{i\notin J} G_i^q.$$ 

One has also that
\begin{equation}\label{***}
\tilde G_J^q = G_J^q \times \bar H_J^q \qquad \text{(almost direct
  product)}.
\end{equation}

Observe that $I(M)_q \subset \tilde G_J^q$, since the full isotropy at
$q$ leaves invariant $L_J(q)$.  Then,  

\begin{equation}\label{isotropy decomposition}
\Lie(I(M)_q) = \mathfrak{k}_J^q \oplus \bar {\mathfrak{h}}_J^q \qquad
\text {(direct sum of ideals)}, 
\end {equation}
where $\bar{\mathfrak{h}}_J^q = \Lie(\bar H_J^q)$ and
$\mathfrak{k}_J^q = \bigoplus _{j\in J} \mathfrak{k}_j^q$.

\subsection{The bracket formula}

We recall the Koszul formula, which gives the Levi-Civita connection
in terms of the Riemannian metric tensor and the Lie bracket: 
\begin{align*}
2\langle\nabla_XY, Z\rangle & = X\langle Y, Z\rangle + Y\langle X,
Z\rangle - Z\langle X, Y\rangle \\  
& \qquad + \langle[X, Y], Z\rangle - \langle[X, Z], Y\rangle -
\langle[Y, Z], X\rangle.
\end{align*}

Assume now that $X, Y, Z$ are Killing fields.  Since the flow of a
Killing field preserves the metric tensor, then the Lie derivative of
the metric along this Killing field is zero.  Then 
$$X\langle Y, Z\rangle = \langle[X, Y], Z\rangle + \langle Y, [X,
  Z]\rangle$$ 
and the same if one permutes $X, Y, Z$.  By making use of these
relations in the Koszul formula one obtains the well known formula for
the Levi-Civita connection in terms of Killing fields: 
\begin{equation}\label{koszul}
2\langle\nabla_XY, Z\rangle = \langle[X, Y], Z\rangle + \langle[X, Z],
Y\rangle + \langle[Y, Z], X\rangle. 
\end{equation}

Assume now that $Y$ is parallel at $q$.  Then
$$\langle[X, Y], Z\rangle_q + \langle[X, Z], Y\rangle_q + \langle[Y,
  Z], X\rangle_q = 0.$$

\begin{prop}\label{bracket formula}
Let $M$ be a homogeneous Riemannian manifold and let $\mathcal L$ be its
foliation of symmetry (i.e., the elements of $\mathcal L$ are the
integral ma\-ni\-folds of the distribution of symmetry $\mathfrak s$).
Assume that the metric of $M$ projects down to $M/\mathcal L$ (only
locally, since  $M/\mathcal L$ could be not a manifold if the elements
of $\mathcal L$ are not closed submanifolds).  Let $q \in M$ and let
$X$ be a Killing field which is parallel at $q$ (observe that $X(q)
\in T_qL(q)$, where $L(q)$ is the element of $\mathcal L$ containing
$q$).  Let $\xi, \eta$ be Killing fields of $M$ such that their
restriction to $L(q)$ is always perpendicular to $L(q)$.  Then  
$$\langle[\xi, X ], \eta\rangle_q = -\frac{1}{2}\langle X, [\xi,
  \eta]\rangle_q.$$ 
\end{prop}

\begin{proof} 
In the equality previous to the proposition, rename $Y$ by $X$, $X$ by
$\xi$ and $Z$ by $\eta$.  So,
\begin{equation}\label{*}
\langle[\xi, X], \eta\rangle_q + \langle[\xi, \eta ], X\rangle_q +
\langle[X, \eta], \xi\rangle_q = 0. 
\end{equation}

Now observe that 
\begin{equation}\label{**}
0 = X(q)\langle\xi, \eta\rangle = \langle\nabla_X\xi, \eta\rangle_q +
\langle\xi, \nabla_X\eta\rangle_q, 
\end{equation} 
since the metric projects down to the quotient (observe that the
Killing fields are projectable since their flow preserve the
foliation of symmetry). 

Observe that, since $\nabla$ is torsion free and $(\nabla X)_q = 0$,
then $(\nabla_X\xi)_q = [X, \xi]_q$ and $(\nabla_X\eta)_q = [X,
  \eta]_q$.  So, from equality \ref{**}, one obtains that $\langle[X,
  \xi], \eta\rangle_q + \langle\xi, [X, \eta]\rangle_q = 0$.  Thus,
$\langle[X, \eta], \xi\rangle_q = \langle[\xi, X], \eta\rangle_q$.
Therefore, by making use of this equality in \ref{*}, one gets the
desired formula. 
\end{proof}

\begin{nota}\label{Fi}
Proposition \ref{bracket formula} remains true is one replace
$\mathcal L$ by the foliation $\mathcal L_i$, whose leaves are
$L_i(x)$.  If $i > 0$, then the quotient is a manifold since the
leaves are compact and hence the orbit of a compact Lie group.  
\end{nota}

\section{Proof of Theorem \ref{main1}}

\begin{assumption}
In the following we will assume that $M = G/H$ is a compact normal
homogeneous Riemannian manifold, which is locally irreducible and not
locally symmetric ($G$ compact and connected). 
\end{assumption}

If the isotropy $H$, let us say at $q$, has fixed non-zero vectors on
$T_qM$, then the full (connected) isometry group $I_0(M)$ is in
general bigger than $G$.  If $G \subsetneq I_0(M)$, then the metric on
$M$ is not anymore normal homogeneous with respect to the presentation
$M = I_0(M)/(I_0(M))_q$.  Otherwise, the transvection group of the
canonical connection would be $I_0(M)$.  A contradiction (see Section
\ref{preliminaries}).
 
A transvection $X$ at $q$, i.e.\ a Killing field which is parallel at
$q$, may not lie in $\mathfrak g$.  In fact, as it will follow from our
main result, it will never lie in $\gg$ (unless the fixed vectors of
$H$ coincides with $T_q(G_0 \cdot q)$, where $G_0$ is the abelian part of
$G$). 

Since $M$ is homogeneous, the distribution $x \mapsto \mathfrak s_x$
is $G$-invariant and thus non-singular and $C^\infty$.  In particular,
$i_{\mathfrak s}(M)= \dim(\mathfrak s_q)$.  Observe that $\mathfrak
p^{g \cdot q} = \Ad(g)\mathfrak p^q $ and $\mathfrak k^{g \cdot q} =
\Ad(g)\mathfrak k^q$.

\begin{lema}\label{factor} 
Let $\Sigma(q)$ be the connected component of the fixed points of $H$
in $M$ (or, equivalently, $\Sigma(q)$ is the integral manifold by $q$
of the distribution $\mathcal D$ of fixed vectors of the isotropies).
Then $\Sigma(q)$ is a local factor of $L(q)$. 
\end{lema}

\begin{proof}
From \cite{R}, one has that the full connected group of isometries
$I_0(M)$ of $M$ is $G_{\mathrm{ss}} \times F$, where $G_{\mathrm{ss}}$
is the semisimple part of $G$ and the Killing fields induced by $F$
are the $G$-invariant vector fields (such a field is determined
uniquely by a vector in $\mathcal D_q$).  Then the flow of any Killing
field induced by $F$ preserve the (autoparallel) distribution
$\mathcal D$ of fixed vectors of the isotropies.  Observe that
$G_{\mathrm{ss}}$ also preserves $\mathcal D$, since $G$ does so.  So,
$\mathcal D$ is preserved by $I_0(M)$.  In particular, $\mathcal
D|_{L(q)}$ is preserved by the transvection group $G^q \subset
I_0(M)$ of $L(q)$.  Then $\mathcal D|_{L(q)}$ is a parallel
distribution of $L(q)$, which contains, by Lemma \ref{flat fixed}, the
whole flat factor $L_0(q)$.  This implies the assertion.
\end{proof}

\begin{lema}\label{flat fixed}
Let $L_0(q) = G_0(q) \cdot q$ be the flat part of $L(q)$.  Then $T_q(L
_0(q))$ is included in the set of fixed vectors in $T_qM$ of $H$.  
\end{lema}

\begin {proof}
Let $X$ be a Killing field induced by $H$.  Then $X$ is bounded, since
$M$ is compact.  The one-parameter group of isometries induced by $X$
must leave $L_0(q)$ invariant.  So, $X|_{L_0(q)}$ is always tangent to
$L_0(q)$.  Since $X$ is bounded and $L_0(q)$ is flat, then
$X|_{L_0(q)}$ must be parallel.  Since $X(q) = 0$, then $X|_{L_0(q)} =
0$.  The assertion follows since $H$ is connected. 
\end {proof}

\begin{lema}\label{G}
Let $M = G/H$ be compact locally irreducible normal homogeneous space
which is not locally symmetric.  Let $L_i(q)$ be a de Rham factor by
$q$ of $L(q)$ which is perpendicular at $q$ to the factor $\Sigma(q)$.
Let $X$ be a transvection at $q$ lying in $\mathfrak g_i^q =
\Lie(G^q_i)$.  Then $X$ lies in $\mathfrak g = \Lie(G)$. 
\end{lema}

Observe that the fact that $L_i(q)$ is perpendicular to $\Sigma(q)$ is
equivalent, by Lemma~\ref{factor} and Lemma \ref{flat fixed}, to
$L_i(q)$ not being contained in $\Sigma (q)$.  Observe that $i\geq 1$,
by the previous lemma. 

\begin{proof}[Proof of Lemma \ref{G}]
We have that $I_0(M) = G_{\mathrm{ss}} \times \tilde\Sigma $ where
$\tilde\Sigma$ is $\Sigma(q)$, but regarded as a Lie group:  the
Killing fields on $M$ induced by  $\tilde\Sigma(q)$ are the $G$
invariant fields, which are determined by its initial condition, which
is a vector in $T_q(\Sigma (q))$ (see Section \ref{preliminaries}).
Let $\tilde\Sigma_{\mathrm{ss}}$ be the semisimple part of the Lie
group $\tilde\Sigma $.  We need the following auxiliary result. 

\begin{sublemma}
$\tilde\Sigma_{\mathrm{ss}} \subset G_J^q$, where $J = \{j \in \{1,
  \ldots, s\}: L_j(q) \subset \Sigma_{\mathrm{ss}}(q)\}$ and
  $\Sigma_{\mathrm{ss}}(q)$ is the semisimple local factor of the
  symmetric space $\Sigma (q)$. 
\end{sublemma}

\begin{proof}
We have that $\tilde\Sigma_{\mathrm{ss}}$ is a semisimple normal subgroup of
$$\tilde G_J^q = G_J^q \times \bar H_J^q \qquad \text{(almost direct
  product)},$$ 
since it leaves $\Sigma_{\mathrm{ss}}(q)$ invariant and it is a normal
subgroup of $I(M)$.  Let us show that $\tilde\Sigma_{\mathrm{ss}} \cap
\bar H_J^q $ is discrete.  In fact, if $X$ is a Killing field in the
Lie algebra of this group, then $X(q) = 0$, since  $\bar H_J^q$ is
contained in the isotropy $(I(M))_q$.  But a Killing field in the Lie
algebra of $\tilde\Sigma_{\mathrm{ss}} \subset \tilde\Sigma$ is
determined by its value at $q$.  Then $X = 0$ and therefore
$\tilde\Sigma_{\mathrm{ss}} \subset G_J^q$ (we are using here that
$\tilde\Sigma_{\mathrm{ss}}$ has no abelian part). 
\end{proof}

We continue with the proof of Lemma \ref{G}. By the above sublemma,
$\tilde\Sigma_{\mathrm{ss}} \subset G_J^q$.  But, since $[\mathfrak
  g_i^q, \Lie(G_J^q)] = \{0\}$ and $\mathfrak g_i^q \cap \Lie(G_J^q) =
\{0\}$, one has that $[\mathfrak g_i^q, \Lie(\Sigma_{\mathrm{ss}})] =
\{0\}$ and $\mathfrak g_i^q \cap \Lie(\Sigma_{\mathrm{ss}}) = \{0\}$.
This implies that 
$$\mathfrak g_i^q \subset \Lie(G_{\mathrm{ss}}) \oplus \mathfrak a
\subset \mathfrak g \oplus \mathfrak a \oplus
\Lie(\Sigma_{\mathrm{ss}}) = \Lie(I(M))$$ 
(direct sum of ideals), where $\mathfrak a$ is the abelian Lie algebra
associated to the flat part $\tilde\Sigma_0$ of $\tilde\Sigma$.  Since
$\mathfrak g_i^q$ is semisimple one must have that $\mathfrak g_i^q
\subset \Lie(G_{\mathrm{ss}}) \subset \mathfrak g$.  This shows the
lemma.
\end{proof}

Let $X$ be a transvection at $q$ which lies in $\mathfrak g_i^q =
\Lie(G_i^q)$, where the de Rham factor $L_i(q)= G_i^q \cdot q$ is
perpendicular at $q$ to  the factor $\Sigma(q)$ of $L(q)$.  From
Lemma~\ref{G} we have that $X$ lies in $\mathfrak g = \Lie(G)$.
Choose $Y \in \mathfrak m = \mathfrak h^\bot$ such that $Y \cdot q =
\frac{1}{2} X \cdot q$.  Recall that the linear map, from $\mathfrak m$
into itself, $U \mapsto [W, U]_{\mathfrak m}$ is skew-symmetric for
all $W \in \mathfrak m$ (since $M = G/H$ has a normal homogeneous
metric).  Also observe that any $\xi \in \mathfrak m$, regarded as the
Killing field $x \mapsto \xi \cdot x$, which is perpendicular at $q$ to
$L(q)$ remains always perpendicular to $L(q)$ (see Remark
\ref{always-perpendicular}).  Then, for any arbitrary  $\xi, \eta \in
\mathfrak m$ which are perpendicular to $L_i(q)$ at $q$,
$$\langle[\xi, Y], \eta\rangle_q = -\langle Y, [\xi, \eta
]\rangle_q.$$ 

Then, from Proposition \ref{bracket formula}, one obtains that $Z = X -
Y$ satisfies 
$$\langle[\xi, Z], \eta\rangle_q = 0.$$
This implies that the Killing field $Z$, regarded in the quotient
$M/\mathcal L_i$ of $M$ by the foliation $\mathcal L_i$ is identically
zero (since its initial conditions are both zero).  The quotient
$M/\mathcal L_i$ can be regarded globally, since the leaves $L_i(x)$
are compact. Then, the Killing field $Z$ of $M$ is always tangent to
the leaves of the foliation $\mathcal L_i$.  Observe that $Z \cdot q$ is
an arbitrary vector of $T_q(L_i(q))$, since $Z \cdot q = \frac{1}{2}X
\cdot q$.  Observe that the same argument shows that there is a Killing
field, always tangent to the leaves of the foliation $\mathcal L_i$,
with an arbitrary value in $T_x(L_i(x))$, for any fixed $x \in M$.

Let now $\hat{\mathfrak g}_i$ be the ideal of $\mathfrak g$ which
consists of the Killing field that are always tangent to the leaves of
$\mathcal L_i$.  Let $\hat G_i$ be the Lie normal subgroup of $G$
associated to $\hat{\mathfrak g}_i$.  The orbits of $\hat G_i$, by
what previously observed, are the leaves $L_i(x)$ of the foliation
$\mathcal L_i$. Let $G'_i$ be the Lie subgroup of $G$, associated to
the complementary ideal $\mathfrak {g}'_i :=(\hat{\mathfrak
  g}_i)^\bot$.  We have that 
$$G = \hat G_i \times G'_i \qquad \text {(almost direct product)}$$
and that $G'_i$ acts transitively on the quotient $M/\mathcal L_i =
I_0(M)/\tilde G_i^q$.  Here, we denote by $\tilde G_i^q$ the group
$\tilde G_{\{i\}}^q$, according with the notation given in \ref{tilde
  G_J}. 

Let now $Z \in \hat{\mathfrak g}_i$ be such that its associated
Killing field on $M$ vanishes identically on $L_i(x)$, for some $x \in
M$.  If $y\in M$, then there exists $g' \in G'_i$ such that $g' \cdot~
L_i(x) = L_i(y)$, since $G'_i$ acts transitively on the quotient
$M/\mathcal L _i $.  We may assume, by replacing $x$ by another
element in $L_i(x)$, that $g' \cdot x = y$.  Then 
$$Z \cdot y = Z \cdot (g' \cdot x) = dm_{g'}(\Ad((g')^{-1})(Z)) \cdot
x = dm_{g'}(Z \cdot x) = 0,$$ 
where $m$ denotes de action of $G$ on $M$.  This shows that the
Killing field associated to $Z$ vanishes identically.  So, $Z = 0$.
  
\begin{lema}
We have that the normal subgroup $\hat G_i$ of $G$ is contained in the
transvection group $G_i^x$ of $L_i(x)$, for any $x \in M$. 
\end{lema}

\begin{proof}
If $x\in M$, then
$$\hat G_i \subset \tilde G_i^x = G_i^x \times \bar H_i^x,$$ 
since $\hat G_i$ leaves invariant any leaf of $\mathcal L_i$.  We have
that $\hat G_i \cap \bar H_i^q$ is discrete.  Otherwise, there would
exists a $Z \in \hat{\mathfrak g}_i$ which vanishes identically on
$L_i(x)$ and therefore, as previously observed, $ Z = 0$.  The
conclusion follows, by making use that $\hat G_i$ is a normal subgroup
of $\tilde G_i^x$ and the fact that $G_i^x$ is semisimple. 
\end{proof}

\begin{prop}\label{Gfactor}
The locally symmetric (irreducible) space $L_i(x)$ is of group
type. Namely, $G_i^x = K \times K$ (almost direct product), where $K$
is a simple Lie group (of compact type).  Moreover, $\hat G_i$
coincides with one of the factors $K$. 
\end{prop}

\begin{proof}
Assume first that $\hat G_i = G_i^x$.  Then any transvection $X \in
\mathfrak p_i^x \subset \mathfrak g _i^x$ at $x$ lies in
$\hat{\mathfrak g}_i$.  Let now $X_1, \ldots, X_r \in \mathfrak p_i^x$
(i.e.\ transvections at $x$) be such that $X_1 \cdot x, \ldots, X_r
\cdot x$ is a basis of $T_x(L_i(x))$.  Then $X_1, \ldots, X_r$ is a
local trivialization of the distribution $\mathfrak s_i$ of tangent
spaces to the leaves of $\mathcal L_i$.  Since $X_1, \ldots, X_r$ are
parallel at $x$ we conclude that the distribution $\mathfrak s_i$ is
parallel at $x$.  Since $x \in M$ is arbitrary we have that the
(non-trivial) distribution $\mathfrak s_i$ is parallel and so $M$
locally splits since $\dim {\mathfrak s_i} < n = \dim M$ because $M$
is not symmetric.  A contradiction. 

So, $\hat G_i$ is a proper, non-trivial, normal subgroup of $G_i^x$.
Since $L_i(x)$ is a locally irreducible symmetric space the isometry
group $G_i^x$ is semi-simple.  But it cannot be simple, since it has a
proper normal subgroup.  Then $L_i(x)$ must be of group type and
$G_i^x = K \times K$ (almost direct product), with $K$ simple, since
$L_i(x)$ is locally irreducible. 
\end{proof}

\begin{nota}\label{diagonal}
Recall that $L_i(x)$ is Lie group with a bi-invariant metric (since it
is a globally symmetric space of the group type). 
\end{nota}

\begin{nota}
Let us fix $q \in M$. Let $\Psi: \hat{\mathfrak g}_i \to \mathfrak
g_i'$ be defined as follows.  Given $X \in \hat{\mathfrak g}_i$,
$\Psi(X)$ is the unique element of $\mathfrak g_i'$ such that $\Psi(X)
\cdot q = -X \cdot q$.  We get that $\Psi: \hat{\mathfrak g}_i \to
\Psi(\hat{\mathfrak g}_i)$ is an isomorphism of Lie algebras.
Moreover, $\mathfrak g_i^q = \hat{\mathfrak g}_i \oplus
\Psi(\hat{\mathfrak g}_i)$ and
$$\mathfrak k_i^q = \{X + \Psi(X): X \in \hat{\mathfrak g}_i\} =:
\diag(\hat{\mathfrak g}_i \oplus \Psi(\hat{\mathfrak g}_i)).$$ 
\end{nota}

Let $\mathfrak{h}'$ be the ideal of $\mathfrak{h} =\Lie (H) \subset
\mathfrak {g} $ which consists of the Killing fields that vanish
identically on the leaf $L_i(q)$.  We have, from  Lemma~\ref{G}, that
$\mathfrak k_i^q \subset \mathfrak h$.  Then with the same arguments
as in \ref{isotropy decomposition}, for the case where $\mathfrak{h}$
is the full isotropy algebra, we can decompose 
$$\mathfrak h = \mathfrak h' \oplus \mathfrak k_i^q \qquad
\text{(direct sum of ideals)}.$$ 

Since $\mathfrak k_i^q$ is simple, from Remark \ref {bilinear}, the
above decomposition must be orthogonal with respect to the $\Ad
(G)$-invariant inner product $\langle\cdot, \cdot\rangle$ on
$\mathfrak g$ (or, more generally, with respect to any $\Ad
(G)$-invariant symmetric bilinear form on $G$). 

Since $\mathfrak h' \subset \bar{\mathfrak h}_i^q$ and $\bar{\mathfrak
  h}_i^q$ is an ideal of $\tilde{\mathfrak g}_i^q$, we get that
$\mathfrak h'$ is an ideal of $\mathfrak h' \oplus \mathfrak g_i^q$.
This decomposition, since $\mathfrak g_i^q$ is semisimple, must be
orthogonal by Remark \ref{bilinear}. 

On the other hand, the decomposition $\mathfrak g = \mathfrak h'
\oplus \mathfrak k_i^g \oplus \mathfrak m$ is also orthogonal.  So, we
get that $(\mathfrak k_i^q)^\bot \cap \mathfrak g_i^q \subset
\mathfrak m$. 

Let us define
$$\mathfrak m_1 := (\mathfrak k_i^q)^\bot \cap \mathfrak g_i^q \subset
\mathfrak m  \qquad \text{and} \qquad \mathfrak m_2 := (\mathfrak
m_1)^\bot \cap \mathfrak m.$$

Keeping the notation of Proposition \ref{Gfactor} and Remark
\ref{diagonal} we have that the decomposition $\mathfrak g_i^q =
\hat{\mathfrak g}_i \oplus \hat{\mathfrak g}_i$, where the first
summand is regarded as an ideal of $\mathfrak g$ and the second one as
a subalgebra of $\mathfrak g_i'$, must be orthogonal (see Remark
\ref{bilinear}).  Moreover, just by taking a positive multiple of the
$\Ad(G)$-invariant inner product $\langle\cdot, \cdot\rangle$ on
$\mathfrak g$, we may assume that the restriction of $\langle\cdot,
\cdot\rangle$ to $\mathfrak g_i^q = \hat{\mathfrak g}_i \oplus
\hat{\mathfrak g}_i$ has the form
$$\langle\cdot, \cdot\rangle|_{\mathfrak g_i^q} = B \oplus \lambda B$$
for some $\lambda > 0$, where $B$ is the negative of the Killing form
of $\hat{\mathfrak g}_i$. 

\medbreak

With these notation and identifications we have 
$$\mathfrak k_i^q = \{(v,v):  v \in \hat{\mathfrak g}_i\}, \qquad
\mathfrak p_i^q = \left\{\left(\frac{1}{2}v, -\frac{1}{2}v\right): v
\in \hat{\mathfrak g}_i\right\}$$
and
$$\mathfrak m_1  = \left\{\left(\frac{-\lambda}{1 + \lambda}v,
\frac{1}{1 + \lambda}v\right): v \in \hat{\mathfrak g}_i\right\}.$$ 

With this notation, if $X = (u, v) \in \mathfrak g_i^q$, then
$$X \cdot q = (u,v) \cdot q = (u - v) \cdot q.$$

Let  $a = \frac{1}{2}(\lambda - 1)$.  Then, for all $v \in
\hat{\mathfrak g}_i$,
$$\left(\left(\frac{1}{2} + a\right)v, -\frac{1}{2}v\right) \in
\mathfrak m_1.$$
In fact, if $(w, w) \in \mathfrak k_i^q$, then
\begin{align*}
\left\langle(w, w), \left(\left(\frac{1}{2} + a\right)v,
-\frac{1}{2}v\right)\right\rangle & = \left\langle w,
\left(\frac{1}{2} + a\right)v\right\rangle + \lambda\left\langle w,
-\frac{1}{2}v\right\rangle \\
& = \left(\frac{1}{2} + a - \frac{\lambda}{2}\right)\langle w,
v\rangle \\
& = \left(\frac{1}{2} + \frac{1}{2}(\lambda - 1) -
\frac{\lambda}{2}\right)\langle w, v\rangle = 0. 
\end{align*}

Note that 
$$\left(\left(\frac{1}{2} + a\right)v, -\frac{1}{2}v\right) \cdot q = (1 +
a)v \cdot q. $$

Observe, identifying $T_qM \simeq \mathfrak m$, that 
$$\mathfrak m_2 = (T_q(L_i(q)))^\bot.$$

Let now $\xi, \eta \in \mathfrak m_2$.  As we have previously observed,
the Killing fields on $M$ induced by $\xi, \eta$ are always
perpendicular to $L_i(q)$.  If $X = (\frac{1}{2}v, -\frac{1}{2}v) \in
\mathfrak p_i^q$ is a transvection at $q$, then, by Proposition
\ref{bracket formula},
$$\langle[\xi, X], \eta\rangle_q = -\frac{1}{2}\langle X, [\xi,
  \eta]\rangle_q.$$

Since $\hat {\mathfrak g}_i$ is an ideal,
$$\langle[\xi, (av, 0)], \eta\rangle_q = 0.$$

Then, if $Y = X + (av, 0)$,
$$\langle[\xi, Y], \eta\rangle_q = \langle[\xi, X], \eta\rangle_q =
-\frac{1}{2}\langle X, [\xi, \eta]\rangle_q = -\frac{1}{2}\langle X
\cdot q, [\xi, \eta] \cdot q\rangle.$$
Since $X \cdot q = \frac{1}{1 + a}Y \cdot q$, the above equality yields
$$\langle[\xi, Y], \eta\rangle_q = -\frac{1}{2(1 + a)}\langle Y \cdot q,
[\xi, \eta] \cdot q\rangle.$$
But on the other hand, since $Y \in \mathfrak m_1 \subset \mathfrak
m$, one must have that
$$\langle[\xi, Y], \eta\rangle_q = -\langle Y \cdot q, [\xi, \eta]
\cdot q\rangle.$$
Then, either $\frac{1}{2(1 + a)} = 1$, or $ \mathfrak m_1$ is
perpendicular to $[\mathfrak m_2 , \mathfrak m_2]_{\mathfrak m}$.  In
the first case
$$1 = \frac{1}{2(1 + \frac{1}{2}(\lambda -1 ))}  = \frac{1}{1 +
  \lambda}.$$
A contradiction, since $\lambda >0$.  In the second case the
distribution $(\mathfrak s_i)^\bot$ is integrable and hence totally
geodesic.  Also a contradiction (see Remark \ref{s_i-bot-integrable}).

Then there does not exist a transvection at $q$ which is perpendicular
at $q$ to the vectors fixed by the isotropy. This proves Theorem \ref
{main1}.

\begin{nota}\label{s_i-bot-integrable}
We denote by $\tilde\xi$ the Killing field on $M$ induced by $\xi$,
which is given by $\tilde\xi_x = \xi \cdot x$ (and the same for $\eta$).
We have that $[\xi, \eta]_{\mathfrak m} \in \mathfrak m_2$, and we
want to see that $[\tilde\xi, \tilde\eta]$ lies in the distribution
$(\mathfrak s_i)^\bot$.  Recall that if $q' = g \cdot q$ and $\bar\xi,
\bar\eta$ are the right-invariant fields on $G$ with $\bar\xi_e =
\xi_e$ and $\bar\eta_e = \eta_e$, then $\tilde\xi_{q'} =
d\pi(\bar\xi_g)$, $\tilde\eta_{q'} = d\pi(\bar \eta_g)$ and
$[\tilde\xi, \tilde\eta]_{q'} = d\pi([\bar\xi, \bar\eta]_g)$, since
$[\bar\xi, \bar\eta]$ is right-invariant (here $\pi: G \to M$ is the 
projection map).

On the other hand, we may identify $T_{q'}M$ with $\Ad(g^{-1}\mathfrak
m)$ via $Z \mapsto Z \cdot q'$. Note that the isotropy at $q'$ is
$\Ad(g^{-1})\mathfrak h$ and $(\mathfrak s_i)^\bot_{q'} =
(\Ad(g^{-1})\mathfrak m_2) \cdot q'$.  So,
\begin{align*}
[\tilde\xi, \tilde\eta]_{q'} & = d\pi([\bar\xi, \bar\eta]_g) =
d\pi(dr_g[\bar\xi, \bar\eta]_e) \\
& = -d\pi(dr_g[\xi, \eta]_e) = -d\pi(dl_g\Ad(g^{-1})[\xi, \eta]_e) \\
& = -dl_gd\pi(\Ad(g^{-1})[\xi, \eta]_e) \in (\mathfrak s_i)^\bot_{q'},
\end{align*}
where we denote by $l_g$ the left-multiplication in $G$ and $M$.
Therefore, $(\mathfrak s_i)^\bot$ is integrable and autoparallel.
\end{nota}

\begin{nota}\label{always-perpendicular}
Let $\xi \in \mathfrak m$ such that the Killing field $\tilde\xi$
defined by $\tilde\xi_x = \xi \cdot x$ is perpendicular to $L(q)$ at $q$.
Recall that $T_xL(q) = \mathfrak s_x$ for all $x \in L(q)$.  Let
$\mathfrak m_0 \subset \mathfrak m$ be the subspace such that
$\mathfrak s_q = \mathfrak m_0 \cdot q$.  Since the distribution of
symmetry is $G$-invariant, we have that $\mathfrak s_{q'} =
(\Ad(\mathfrak m_0)) \cdot q'$, where $g \in G$ and $q' = g \cdot q$.  Now,
$\tilde\xi_{q'} = (\Ad(g)\xi) \cdot q'$ is perpendicular to $\mathfrak
s_{q'}$.  Since $g$ is arbitrary, we conclude that if $\tilde\xi$ is
perpendicular to $L(q)$ at $q$, it is always perpendicular to $L(q)$.
\end{nota}

\begin{ejemplo}[Stiefel manifolds]\label{Stiefel} 
Let us consider the Stiefel manifold $M = \SO(n + k)/\SO(n)$, with $n,
k \ge 2$, endowed with the normal homogeneous metric.  It is well
known, as it follows from the Serre long exact sequence of homotopies,
that $M$ is simply connected. Moreover, $M$ is an irreducible
Riemannian manifold (see Remark \ref{irred}).  Then $M$ has index of
symmetry $i_{\mathfrak s}(M) = \frac{1}{2}k(k - 1)$.  Moreover, the
(connected) set of fixed points of $\SO(n)$ in $M$, which contains the
orthogonal $k$-frame $B = (e_1, \ldots, e_k)$, is isomorphic to
$\SO(k)$.  Here we consider the standard inclusions $\SO(n) \simeq  
\left(\begin{smallmatrix}
I_k & 0 \\
0 & \SO(n)
\end{smallmatrix}\right)$
and $\SO(k) \simeq 
\left(\begin{smallmatrix}
\SO(k) & 0 \\
0 & I_n
\end{smallmatrix}\right)$
inside $\SO(n + k)$.  Thus, the leaf of symmetry $L(B)$ is the
symmetric space of the group type $\SO(k) \simeq (\SO(k) \times
\SO(k))/\SO(k)$.
\end{ejemplo}

\begin{nota}\label{irred} 
Let $M = G/H$ be a simply connected naturally reductive Riemannian
manifold, where $G$ is the group of transvections of the canonical
connection.  If $M$ is a normal homogeneous space, $G$ must always
coincides with the group of transvections, provided $G$ acts
effectively.  Let $M = M_0 \times M_1 \times \cdots \times M_r$ be the
de Rham decomposition ($M_0$ is, eventually, trivial).  Then  $G = G_0
\times G_1 \times \cdots \times G_r$ and $H = H_0 \times H_1 \times
\cdots \times H_r$, where $H_i \subset G_i$ and $M_i = G_i/H_i$ for
all $i = 0, 1, \ldots, r$.  It is not easy to find this decomposition
through the mathematical literature.  We found such a fact in
\cite{EM}, but without proof.  So, we next give the argument for such
a decomposition. 

Let $p = (p_0, \ldots, p_r) = eH \in G/H $ and let $\mathfrak g =
\mathfrak m \oplus \mathfrak h$ be the naturally reductive
decomposition associated to the naturally reductive metric on $M$.
Observe, as it is well known, since $G$ is the group of transvections,
that $\mathfrak g = \mathfrak m + [\mathfrak m, \mathfrak m]$.

Identifying $T_pM \simeq \mathfrak m$, let $\mathfrak m = \mathfrak
m_0 \oplus \mathfrak m_1 \oplus \cdots \oplus \mathfrak m_r$, where
$\mathfrak m_i = T_{p_i}M_i$ and $T_pM = T_{p_0}M_0 \oplus T_{p_1}M_1
\oplus \cdots \oplus T_{p_r}M_r$.

If $Z \in \mathfrak g$, let $\hat Z$ be its associated Killing field
$q \mapsto Z \cdot q$. It is well known that $(\nabla\hat Z)_p$ lies
in the normalizer of the holonomy algebra at $p$.  So, $(\nabla\hat
Z)_p$ leaves invariant the tangent space at $p$ of any de Rham
factor. Namely, for all $i = 0, 1, \ldots, r$,
\begin{equation}\label{EM*}
(\nabla_{\mathfrak m_i}\hat Z)_p \subset \mathfrak m_i. 
\end{equation}

If $X \in \mathfrak m$, then
\begin{equation}\label{EM**}
(\nabla\hat X)_p = D_X,
\end{equation}
where $D = \nabla - \nabla^c$ is the difference tensor between the
Levi-Civita connection and the canonical connection.  Moreover,
$\langle D_XY, Z\rangle$ is a $3$-form.  From \ref{EM*} one has that,
if $i \neq j$,
$$D_{\mathfrak m_i}\mathfrak m_j = \{0\}.$$

Let now $X_i \in \mathfrak m_i$, then $dl_{\Exp(tX_i)}$ gives the
$\nabla^c$-parallel transport along the geodesic $\Exp(tX_i) \cdot p
\in M_i$.  If $u \in (\mathfrak m_i)^\bot$, then from \ref{EM**},
$dl_{\Exp(tX_i)}u$ is also parallel with respect to the Levi-Civita
connection, along $\Exp(tX_i) \cdot p$.

This implies that $l_{\Exp(tX_i)}$ acts trivially on any other de Rham
factor $M_j$, $j \neq i$.

Then, if we define $G_i$ to be the Lie subgroup of $G$ whose Lie
algebra is generated by $\mathfrak m_i$ and $H_i = (G_i)_p$ we obtain
the desired decomposition.
\end{nota}

\section{The naturally reductive case}

In this section we assume that $M = G/H$ is a compact naturally
reductive space, with reductive decomposition $\mathfrak g = \mathfrak
h \oplus \mathfrak m$.  We also assume that $M$ is locally irreducible
and non-locally symmetric, and the presentation $G/H$ is given by the
transvection group of the canonical connection of $M$.  That is to
say, $\mathfrak g = \mathfrak m + [\mathfrak m, \mathfrak m]$ (which
is not, in general, a direct sum).  From a well known result due to
Kostant \cite{K}, there exists an $\Ad(G)$-invariant non-degenerate
bilinear form $Q$ on $\mathfrak g$ such that
$$Q(\mathfrak h, \mathfrak m) = 0, \qquad Q|_{\mathfrak m} =
\langle\cdot, \cdot\rangle,$$
where we denote by $\langle\cdot, \cdot\rangle$ the Riemannian metric
at $T_{eH}M \simeq \mathfrak m$.  In particular, it follows that $Q$
is non-degenerate when restricted to $\mathfrak h$.

We keep the notation from the previous section. More precisely, if
$L(q)$ is the leaf of symmetry by $q \in M$, the de Rham decomposition
of $L(q)$ is given by
$$L(q) = L_0(q) \times L_1(q) \times \cdots \times L_s(q),$$
and the leaf of fixed point of the isotropy $\Sigma(q)$ by $q$ is a
local factor of $L(q)$.  That is, there exist a subset $J \subset \{0,
1, \ldots, s\}$, which must contain $0$, such that 
$$\Sigma(q) = L_J(q) = \prod_{j \in J}L_j(q).$$

If $i \notin J$, we proved that $L_i(q)$ is a symmetric space of the
group type.  Moreover, the ideal $\hat{\mathfrak g}_i \subset
\mathfrak g$, which consist of the Killing fields which are  always
tangent to the foliation with leaves $L_i(q)$, $x \in M$, is a simple
ideal of the Lie algebra of geometric transvections $\mathfrak g_i^q = \mathfrak
k_i^q \oplus \mathfrak p_i^q$.  In particular, there exists an ideal
$\mathfrak k$ contained in the $Q$-orthogonal ideal to $\hat{\mathfrak
  g}_i$, which turns out isomorphic to $\hat{\mathfrak g}_i$, such
that $\mathfrak k_i^q \simeq \diag(\hat{\mathfrak g}_i \oplus
\hat{\mathfrak g}_i)$.

Recall that $Q|_{\mathfrak g_i^q}$ has the form $\lambda B \oplus \mu
B$, where $B$ is the negative of the Killing form of $\hat{\mathfrak
  g}_i$ and we identify $\mathfrak g_i^q \simeq \hat{\mathfrak g}_i
\oplus \hat{\mathfrak g}_i$.  We have that $\lambda$ and $\mu$ are
both nonzero, since $Q|_{\mathfrak h}$ is non-degenerate.

If $Q$ is not positive defined, then $\lambda, \mu$ cannot have both
the same sign.  If $\lambda > 0$, we get a contradiction with the same
argument that we use for the normal homogeneous case (because, we only
use that $\mu \neq 0$).  Therefore, it only remains the case $\lambda
< 0 < \mu$.  By rescaling the Riemannian metric on $M$, we may assume
that
$$Q|_{\mathfrak g_i^q} = -\lambda B \oplus B, \qquad \text{for some
  $\lambda > 0$}.$$
Recall that $\lambda \neq 1$, since $Q$ is non-degenerate on the
isotropy $\mathfrak h$.  Let, $\mathfrak m_1 \subset \mathfrak m$
the subspace such that $\mathfrak m_1 \cdot q = T_q(L_i(q))$. So,
with our identifications, we have that
$$\mathfrak k_i^q = \{(v,v):  v \in \hat{\mathfrak g}_i\}, \qquad
\mathfrak p_i^q = \left\{\left(\frac{1}{2}v, -\frac{1}{2}v\right): v
\in \hat{\mathfrak g}_i\right\}$$
and
$$\mathfrak m_1  = \left\{\left(\frac{1}{1 - \lambda}v,
\frac{\lambda}{1 - \lambda}v\right): v \in \hat{\mathfrak
  g}_i\right\}.$$

In this case, we have that if $a = -\frac{1}{2}(\lambda + 1)$, then
$$\left(\frac{1}{2}v, -\frac{1}{2}v\right) + (av, 0) \in \mathfrak m_1$$
for all $v \in \hat{\mathfrak g}_i$. Just by following the
calculations from the last part of the previous section we get the
contradiction
$$1 = \frac{1}{2(a + 1)} = \frac{1}{1 - \lambda}.$$

This proves Theorem \ref {main2}.

\section{Examples with distribution of symmetry not of a group-type}

For a compact (simply connected) naturally reductive space $M = G/H$
the leaves of the symmetric foliation, are always globally symmetric
spaces of group type.  In fact, if $G$ is the group of transvections
of the canonical connection then the symmetric submanifold $S$ by $p =
[e]$ coincides with the connected component of the fixed points of $H$
in $M$.

The full isometry group is given by (see \cite {R, OR}) 
 $$I_0 (M) = G_\mathrm{ss} \times \tilde S,$$
where $\tilde S$ is the group of isometries which corresponds to the
$G$-invariant vector fields of $M$, and $G_\mathrm{ss}$ is the
semisimple part of the reductive group $G$.  Moreover, the leaves of
the symmetric foliation are given by the orbits on $M$ of the group
$\tilde S$.  This implies that the holonomy group $\Phi$ of a fixed
leaf $S$, of the symmetric foliation, commutes with $\tilde S$.  It is
not hard  to see that $\Phi$ must be isomorphic to $\tilde S$ and that
$I_0(S) =  \Phi \times S \simeq S\times S$.  

In the next we will construct examples of compact (simply connected)
homogeneous spaces whose foliation of symmetry is not of group type.  

Let $G$ be a compact Lie group and let $G', K'$ be compact subgroups
such that  $G\supset G' \supset K'$. Assume, furthermore, that $G'$ is
simple and that $(G',K')$ is a symmetric pair.  Observe that $(G',K')$
is not of group-type, since $G'$ is simple.  

Let $(\cdot, \cdot)$ be an $\Ad(G)$-invariant inner product in the Lie
algebra $\mathfrak {g}$ of $G$. Let 
$$\mathfrak {g}' = \mathfrak {k}'\oplus \mathfrak {p}'$$
be the Cartan decomposition associated to $(G',K')$.  Since $G'$ is
simple, then the restriction of  $(\cdot, \cdot)$ to $\mathfrak{g}'$
is a multiple of the Killing form.  So, $\mathfrak{k}'$ must be
perpendicular to $\mathfrak {p}'$ with respect to $\langle\cdot,
\cdot\rangle$ (since both subspaces are perpendicular with respect to
the Killing form of  $\mathfrak {g}'$).  This is in general false if
$G'$ is not simple (e.g., if $(G',K')$ is of group type). 

Let now 
$$\mathfrak {m} := (\mathfrak {k}')^\perp$$
be the orthogonal complement in $\mathfrak{g}$ with respect to $
(\cdot, \cdot)$.   

As previously observed, $\mathfrak{p}' \subset \mathfrak{m}$. So, if
$\mathfrak{m}' = (\mathfrak{p}')^\perp \cap \mathfrak{m}$, then
$$\mathfrak{m} = \mathfrak{m}' \oplus \mathfrak{p}' \qquad
\text{(orthogonally)}.$$

Let $\langle\cdot, \cdot\rangle$ be the inner product on
$\mathfrak{m}$ defined as follows by the following properties: 

\begin{enumerate} 
\item[{\it a})] $\langle \mathfrak{k}', \mathfrak{p}'\rangle = 0$;
\item[{\it b})] the restrictions to $\mathfrak{m}'$ of $(\cdot,
  \cdot)$ and $\langle\cdot, \cdot\rangle$ coincide; 
\item[{\it c})] $\langle\cdot, \cdot\rangle  = 2(\cdot, \cdot)$,
  restricted to the subspace $\mathfrak{p}'$. 
\end{enumerate}

Let $M : = G/K'$ endowed with the $G$-invariant Riemannian metric that
at $p = [e]$ coincides  with  the inner product $\langle\cdot,
\cdot\rangle$ of $T_pM \simeq \mathfrak{m}$.  We will also denote by
$\langle\cdot, \cdot\rangle$ this Riemannian metric.  The associated
Levi-Civita connection will be denoted by~$\nabla$.

\begin{notation}
If $X \in \mathfrak{g}$ then $\tilde X$ denotes the Killing field of
$M$ induced by $X$. Namely, $\tilde X(q) = \frac{d}{dt}\big|_0\Exp(tX)
\cdot q$.
\end{notation}

Recall, the well known fact that 
$$[\tilde X, \tilde Y] = -\widetilde{[X,Y]}.$$

\begin{lema}\label{nabla}
If $X \in \mathfrak{p}'$ then $(\nabla\tilde X)_p = 0$.  
\end {lema}

\begin{proof}
Let us first show that $G'/K'$ is a totally geodesic submanifold of
$G/K'$. Observe that $\mathfrak{m}'$ is the orthogonal complement,
with respect to $(\cdot, \cdot)$, of $\mathfrak{g}'$ in
$\mathfrak{g}$.  Then $\mathfrak{m}'$ must be $\Ad(G')$-invariant,
since $\mathfrak{g}'$ is so. 

Let $\xi \in \mathfrak{m}'$ and $g' \in G'$.  Since $\mathfrak{g}'$ is
$G'$-invariant,  
$$(\tilde\xi)_{g'K'} = dl_{g'}\Ad((g')^{-1})\xi.$$

Now observe that $\Ad((g')^{-1})\xi$ belongs to $\mathfrak{m}'$ and so
it is perpendicular to $G'/K'$ at $eK'$.  Since $l_{g'}$ is an
isometry that preserves $G'/K'$ one concludes that $(\tilde\xi)_{g'K'}$
is perpendicular to $G'/K'$ at $g'K'$.  Therefore, the Killing field
$\tilde\xi$, restricted to $G'/K'$ is always perpendicular to this
submanifold.  If $A$ is the shape operator of $G'/K'$ then, for any 
$U,V$ vector fields on  $G'/K'$, 
$$\langle A_{\tilde\xi} U, V\rangle = -\langle\nabla_U\tilde\xi,
V\rangle.$$ 
But the left hand of the above equality is symmetric on $U, V$ and the
right hand is skew (by the Killing equation).  Then  
$$A_{\tilde\xi} = 0 = \langle\nabla_U\tilde\xi, V\rangle.$$  
Since $\xi$ is any arbitrary normal direction at $p$, we have that
$G'/K'$ is a totally geodesic submanifold of $M$.  

Since $X \in \mathfrak{p}'$ then $\tilde X|_{G'/K'}$ is parallel at
$p$, regarded as a Killing field of $G'/K'$. Then, since $G'/K'$ is
totally geodesic,  
\begin{equation}\label{(*)}
\nabla_Y\tilde X = 0,
\end{equation}
for all $Y \in \mathfrak{p}' \simeq T_p(G'/K')$.  Observe, from last
equality and the Killing equation, that  
\begin{equation}\label{(**)}
\langle\nabla_\xi\tilde X, Y\rangle = 0,
\end{equation}
for any  $\xi \in \mathfrak{m}',\, Y \in \mathfrak{p}'$.

Let $\xi, \eta \in \mathfrak{m}'$. Then, from equation \ref{koszul}, 
\begin{align}
2\langle \nabla_\xi\tilde X, \eta\rangle & = \langle[\tilde\xi ,
  \tilde X]_p, \eta\rangle + \langle[\tilde\xi, \tilde\eta]_p,
X\rangle + \langle[\tilde X, \tilde\eta]_p, \xi\rangle \notag \\
&  = -\langle[\xi, X], \eta\rangle - \langle[\xi, \eta], X \rangle -
\langle[X, \eta], \xi\rangle \notag \\ 
& = -([\xi, X], \eta) - 2([\xi, \eta], X) - ([X, \eta], \xi) \notag \\
& = (X, [\xi, \eta]) - 2([\xi, \eta], X) + (X, [\xi, \eta]) =
0, \label{(***)}
\end{align}
where $(\cdot, \cdot)$ is the $\Ad(G)$-invariant inner product of
$\mathfrak{g}$.  

From \ref{(*)}, \ref{(**)} and \ref{(***)} one has that $(\nabla\tilde
X)_p = 0$. 
\end{proof}

Let $\tilde{\mathfrak{p}}'$ denote the $G$-invariant distribution on
$M$ with $\tilde{\mathfrak{p}}'_p =  \mathfrak{p}'$.  Then the
distribution $\tilde{\mathfrak{p}}'$ is integrable with totally
geodesic leaves.  In fact, if $g \in G$, then the leaf by $g \cdot p$ of
$\tilde{\mathfrak{p}}'$ is $g \cdot (G'/K')$ (see the beginning of the
proof of Lemma \ref{nabla}, where it is proved that $G'/K'$ is a
totally geodesic submanifold of $M$).  

\begin{lema}\label{examples} 
We keep the notation and general assumptions of this section.  Assume,
furthermore, that $(G,G')$ is  an  irreducible (almost effective)
symmetric pair and that $(G/K', \langle\cdot, \cdot\rangle)$ is not a
symmetric space.  Then the distribution of symmetry of $(G/K',
\langle\cdot, \cdot\rangle)$ is exactly $\mathfrak s =
\tilde{\mathfrak {p}}'$ (whose  integral manifolds are $g \cdot (G'/K')$,
$g\in G$).
\end{lema} 

\begin{proof} 
We have seen that $\tilde{\mathfrak {p}}' \subset \mathfrak s$.  Then
the distribution $\mathfrak s$ of $M= G/K'$ descends to a distribution
$\bar{\mathfrak s}$ of the symmetric space $\bar M = G/G'$. Moreover,
$\bar{\mathfrak s}$ is $G$-invariant, since $\mathfrak s$ is so
(because isometries preserve the distribution of symmetry).  Then
$\bar{\mathfrak s} = 0$, or $\bar{\mathfrak s} = T\bar M$.  In the
first case we have that $\mathfrak s = \tilde{\mathfrak{p}}'$, as we
wanted to prove.  In the second case we obtain that $\mathfrak s = TM$
and hence $M$ would be symmetric.  
\end{proof}

We have not found a general argument for deciding when $(G/K',
\langle\cdot, \cdot\rangle)$ is not locally symmetric.

\subsection{Explicit examples}

We keep the notation and assumptions of this section.  

\begin{ejemplo}[The unit tangent bundle of the sphere of curvature
    $2$]\label{unit tangent bundle}
Set, for $n
\geq 1$,
$$G = \SO(n + 1), \qquad G'= \SO(n), \qquad  K'= \SO(n - 1).$$ 
Let $M = G/K'$ endowed with the $\SO(n + 1)$-invariant metric
$\langle\cdot, \cdot\rangle$.  Though, for $n = 4$, $\SO(n)$ is not
simple, the restriction of the Killing form of $\SO(5)$ to $\SO(4)$
turns out to be a multiple of the Killing form (i.e., the same
multiple on each irreducible factor).  Then the general results of
this section applies also for this case.  We have proved that the
(autoparallel) $\SO(n + 1)$-invariant distribution
$\tilde{\mathfrak{p}}'$ is contained in the symmetric distribution
$\mathfrak s$ of $M$. 

We will prove the equality, or equivalently, from Lemma
\ref{examples}, that
$$(\SO(n + 2)/\SO(n), \langle\cdot, \cdot\rangle)$$
is not a locally symmetric space.  

It is not difficult to see that, up to rescaling, $(\SO(n + 2)/\SO(n),
\langle\cdot, \cdot\rangle)$ is the unit tangent bundle of the sphere
of curvature $2$ (with the Sasaki metric).  From the following remark
we obtain that this space is not locally symmetric.
\end{ejemplo}

\begin{nota}\label{not symmetric}
The unit tangent bundle $M^{2n - 1}$ to the sphere of curvature $2$ is
never locally symmetric.  In fact, assume $n = 2$. Then $\SO(3)$ acts
simply transitively on $M^3$.  The universal cover $\tilde M^3$ of
$M^3$ is compact and diffeomorphic to $\Spin(3)$.  If $\tilde M^3$ is
symmetric then it must be irreducible.  Moreover, $\tilde M^3$ must be
of rank one.  In fact the isotropy representation is an irreducible
and polar representation acting in a $3$-dimensional space (which
implies that is transitive on the unit sphere).  Then all geodesics in
$\tilde M^3$ have are closed of the same length.  Then all geodesics
of $M^3$ would admit a common period.  But any  horizontal geodesic of
$M^3$ has length $\sqrt 2 \pi$.  But a vertical geodesic has length
$2\pi$, which is not rationally related to the length of any
horizontal geodesic.  Then $M^3$ is not locally symmetric.  Now
observe that $M^3$ is in a canonical way a totally  geodesic
submanifold of $M^{2n - 1}$.  Then $M^{2n - 1}$ is not locally
symmetric. Otherwise, $M^3$ would be locally symmetric.

Observe that the unit tangent bundle to the sphere of dimension $2$
and curvature $1$ is a Lie group with a bi-invariant metric and hence
a symmetric space.  
\end {nota}

\begin{nota}
Notice the difference between the above example and taking $k = 2$ in
Example \ref{Stiefel}.  In both cases the underlying differentiable
manifold is the Stiefel manifold $\SO(n + 2)/\SO(n)$, but in Example
\ref{unit tangent bundle} the metric is not normal homogeneous.
Recall that in Example \ref{Stiefel}, with the normal homogeneous
metric, the leaf of symmetry is the circle $S^1$ (and hence, it is of
the group type).
\end{nota}

\begin{ejemplo}
Let us consider the standard inclusions $\SU(3) \supset \SO(3) \supset
\SO(2)$.  Then $M = \SU(3)/\SO(2)$, with the metric above defined has
index of symmetry equal to 2 and leaf of symmetry the sphere $S^2$.
In fact, let us check that $M$ is not a symmetric space.  Recall that
$M$ is the Aloff-Wallach manifold $M = W^7_{1, -1}$, and hence $H^4(M,
\mathbb Z) = 0$, according with \cite[Lemma~3.3]{aloff-wallach-1975}.
Assume that $M$ is a symmetric space, then must be one of the
following: $S^7$, $S^5 \times S^2$, $S^4 \times S^3$, or $S^3 \times
S^2 \times S^2$.  The last two cases are excluded by the restriction on
the cohomology.  It cannot be $M = S^7$, since $\SU(3)$ is not
transitive on $S^7$.  Finally, if $M = S^5 \times S^2$, then,
projecting down to the second factor, we would obtain an isometric 
action of
$\SU(3)$ on $S^2$.  Such an action must be trivial since $\SU(3)$ is
simple and $\dim \SU(3) > \dim \SO(3)$.  So, $\SU(3)$ cannot be
transitive on $M$, which is a contradiction.  With a similar argument,
we can prove that $M$ is an irreducible Riemannian manifold. 
\end{ejemplo}

\begin{ejemplo}[The Wallach $24$-manifold]
Consider $F_4 \supset \Spin(9) \supset \Spin(8)$. Notice that
$F_4/\Spin(9)$ is the Cayley plane.  The manifold $W^{24} =
F_4/\Spin(8)$ is the so-called Wallach manifold and it has leaf of symmetry
$S^8 = \Spin(9)/\Spin(8)$.  In fact, it is well-known that $W^{24}$ is
topologically different from a symmetric space.  Recall that in this
case $W^{24}$ can be endowed with a metric of positive curvature.
\end{ejemplo}

\end{document}